\documentclass[reqno, 12pt]{article}

\pdfoutput=1

\usepackage{enumerate}
\usepackage{latexsym}
\usepackage[centertags]{amsmath}
\usepackage{amsfonts}
\usepackage{amssymb}
\usepackage{amsthm}
\usepackage{newlfont}
\usepackage{graphics}
\usepackage{color}
\usepackage{float}
\usepackage{diagbox}
\usepackage{hyperref}
\usepackage{longtable}
\usepackage{rotating}
\usepackage{multirow}
\usepackage{extarrows}
\usepackage[sort,compress,numbers]{natbib}
\usepackage[utf8]{inputenc}

\newtheorem{thm}{Theorem}[section]
\newtheorem{cor}[thm]{Corollary}
\newtheorem{lem}[thm]{Lemma}
\newtheorem{prop}[thm]{Proposition}

\theoremstyle{mydefinition}

\theoremstyle{myremark}

\newtheorem{exa}[thm]{Example}

\allowdisplaybreaks[4]

\title{On $P$-partitions Extended by Two-Rowed Plane Partitions}

\author{Jingxuan Li$^{1}$, Feihu Liu$^{2}$, and Guoce Xin$^{3}$
\\[2mm]
{\small $^{1, 2, 3}$ School of Mathematical Sciences,}\\[-0.8ex]
{\small Capital Normal University, Beijing, 100048, P.R.~China}\\
{\small $^1$ Email address: jingxuanli27@163.com}\\
{\small $^2$ Email address: liufeihu7476@163.com}\\
{\small $^3$ Email address: guoce\_xin@163.com}
}

\date{December 5, 2024}

\begin{document}

\maketitle

\begin{abstract}
Inspired by Gansner's elegant $k$-trace generating function for rectangular plane partitions, we introduce two novel operators,
$\varphi_{z}$ and $\psi_{z}$, along with their combinatorial interpretations. Through these operators, we derive a new formula for $P$-partitions of posets extended by two-rowed plane partitions. This formula allows us to compute explicit enumerative generating functions for various classes of $P$-partitions. Our findings encompass skew plane partitions, diamond-related two-rowed plane partitions, an extended $V$-poset, and ladder poset extensions, enriching the theory of $P$-partitions.
\end{abstract}

\noindent
\begin{small}
 \emph{Mathematic subject classification}: Primary 05A15; Secondary 05A17, 11P81.
\end{small}

\noindent
\begin{small}
\emph{Keywords}: Two-rowed plane partitions; $P$-partitions; Skew plane partitions; Linear extensions; $q$-shifted factorials.
\end{small}

\section{Introduction}

A \emph{plane partition} is an array $\pi=(\pi_{ij})_{i,j\geq 1}$ of nonnegative integers such that $\pi$ has finitely many nonzero entries and is weakly decreasing in rows and columns. If the sum of all entries, $\sum_{i,j}\pi_{ij}=n$, then $\pi$ is referred to as a plane partition of $n$. Regarding the history of plane partitions, please refer to \cite{AndrewsDream,Bressoud} and \cite[P. 440-442]{RP.Stanley2024}.

A \emph{part} of a plane partition $\pi=(\pi_{i,j})$ is a positive entry $\pi_{i,j}>0$.
In \cite{Stanley-trace}, Stanley defined the \emph{trace} of $\pi$ as $\mathrm{tr}(\pi)=\sum\pi_{i,i}$.
Consider a $r\times c$ matrix $\pi=(\pi_{i,j})$ over nonnegative integers
 such that $\pi_{i,j}\geq \pi_{i,j+1}$ and $\pi_{i,j}\geq \pi_{i+1,j}$. Such a matrix $\pi$ is a plane partition of $n=\sum_{i,j}\pi_{i,j}$ with at most $r$ rows and $c$ columns, and is called a \emph{rectangular plane partition}.
Let $P_{r,c}(X)$ denote the complete generating function for all plane partitions with at most $r$ rows and $c$ columns, i.e.,
$$P_{r,c}(X)=\sum_{\pi}x_{1,1}^{\pi_{1,1}}\cdots x_{1,c}^{\pi_{1,c}}\cdots x_{r,1}^{\pi_{r,1}}\cdots x_{r,c}^{\pi_{r,c}},$$
where the sum ranges over all rectangular plane partitions with at most $r$ rows and $c$ columns.

Setting \(x_{i,j} = x\) for all \(i,j\), we obtain the enumerative generating function \(q_{r,c}(x) := P_{r,c}(X)|_{x_{i,j} = x}\).
When $r,c\rightarrow \infty$, MacMahon \cite{MacMahon78} conjectured the following formula
\begin{align}\label{Qinfty}
q_{\infty,\infty}(x)=\prod_{n=1}^{\infty}\frac{1}{(1-x^n)^n}.
\end{align}
MacMahon \cite{MacMahon97} also conjectured
\begin{align}\label{Q-RC}
q_{r,c}(x)=\prod_{i=1}^{r}\prod_{j=1}^{c} \frac{1}{1-x^{i+j-1}}.
\end{align}
Obviously, \eqref{Q-RC} reduces to \eqref{Qinfty} when $r,c\rightarrow \infty$.
For a proof of a stronger result of \eqref{Q-RC}, see \cite[Sect. 495]{MacMahon19}. A combinatorial proof of \eqref{Q-RC} appears in \cite{Bender72}.

We now outline  a generalization of \eqref{Q-RC} from \cite{Gansner81}.
For an integer $k$ with $-r+1\leq k\leq c-1$, the \emph{$k$-trace} $\mathrm{tr}_k(\pi)$ of $\pi=(\pi_{i,j})$ is defined as
$$\mathrm{tr}_k(\pi)=\sum_{1\leq i\leq r;\ 1\leq j\leq c \atop j-i=k} \pi_{i,j}.$$
For example, the rectangular plane partition
\begin{align*}
\pi=
\begin{matrix}
7&7 &6 & 5 \\
7&5 & 4& 2 \\
6&3 & 0& 0 \\
\end{matrix}
\end{align*}
of $52$ has $\mathrm{tr}_{-2}(\pi)=6$, $\mathrm{tr}_{-1}(\pi)=10$, $\mathrm{tr}_{0}(\pi)=12$, $\mathrm{tr}_{1}(\pi)=11$, $\mathrm{tr}_{2}(\pi)=8$, and $\mathrm{tr}_{3}(\pi)=5$. Define $\tau_{r,c}(t_{-r+1},\ldots, t_{-1}; t_0,\ldots, t_{c-1};n)$ to be the number of such plane partitions of $n$ with $k$-traces $\mathrm{tr}_{k}(\pi)=t_k$, $-r+1\leq k\leq c-1$. Let
\begin{align*}
&T_{r,c}(z_{-r+1},\ldots,z_{-1};z_0,\ldots,z_{c-1}; q)
\\=&\sum_{n=0}^{\infty}\sum_{t_{-r+1}=0}^{\infty}\cdots\sum_{t_{c-1}=0}^{\infty}\tau_{r,c}(t_{-r+1},\ldots, t_{-1}; t_0,\ldots, t_{c-1};n)z_{-r+1}^{t_{-r+1}}\cdots z_{-1}^{t_{-1}}z_0^{t_0}\cdots z_{c-1}^{t_{c-1}}q^n.
\end{align*}

Gansner derived the following elegant result.
\begin{thm}{\em \cite{Gansner81}}\label{Theorem-Gansner}
For positive integers $r,c$, we have
$$T_{r,c}(z_{-r+1},\ldots,z_{-1};z_0,\ldots,z_{c-1}; q)=\prod_{i=1}^r\prod_{j=1}^c \frac{1}{1-z_{-i+1}z_{-i+2}\cdots z_{j-1} q^{i+j-1}}.$$
\end{thm}
In Theorem \ref{Theorem-Gansner}, setting $z_{k}=1$ for all $k$ yields \eqref{Q-RC}, and setting $z_k=1$ for all $k\neq 0$ yields Stanley's trace theorem \cite[Theorem 2.2]{Stanley-trace}.
Andrews and Paule \cite{Andrews12} also proved Theorem \ref{Theorem-Gansner} by using basic power series arithmetic, an approach distinct from Gansner's original proof \cite{Gansner81}, which is based on a combinatorial bijection.

When $r=2$, a rectangular plane partition $\pi$ is known as a \emph{two-rowed plane partition}.
Two-rowed plane partitions have been a subject of extensive research; see, e.g., \cite{Andrews2,Gordon62,Sudler65}.
Inspired by the results related to the $k$-traces, we consider two operators $\varphi_{z}$ and $\psi_{z}$ defined below.

Let $K$ be a ring, such as the ring $\mathbb{Z}$ of integers. Let
$$F(x,y;P)=\sum_{i\geq j\geq 0}a_{ij}x^iy^j,\ \ \ a_{ij}\in K$$
be the bivariate generating function of a combinatorial object $P$.
The linear operators $\varphi_{z}$ and $\psi_{z}$ are defined by
\begin{align}\label{Varphi-map}
\varphi_{z}\odot F(x,y;P)&=\frac{F(xz,yz;P)-yF(xyz,z;P)}{(1-x)(1-y)},\\
\label{Psimap}
\psi_{z}\odot F(x,y;P)&=\frac{F(z,xz;P)-xF(xz,z;P)}{1-x}.
\end{align}

We will prove the following result.
\begin{thm}\label{FullTwoPlane}
Let $\varphi_{z}$ and $\psi_{z}$ be the operators as defined in \eqref{Varphi-map} and \eqref{Psimap}. For $n\geq 2$, we have
\begin{align*}
&\psi_{z_n}\odot \varphi_{z_{n-1}}\odot \varphi_{z_{n-2}}\odot \cdots\odot \varphi_{z_1}\odot F(x,y;P)
\\=&\frac{F(z_1 \cdots z_n, xz_1\cdots z_n;P)-xF(xz_1\cdots z_n,z_1\cdots z_n;P)}{(1-x)(1-xz_n)(1-xz_{n-1} z_n)\cdots(1-xz_2\cdots z_n)(1-z_n)(1-z_{n-1}z_n)\cdots(1-z_2\cdots z_n)}.
\end{align*}
\end{thm}

In Section 3, we will provide a combinatorial interpretation for these two operators. Within this framework, the generating function in Theorem \ref{FullTwoPlane}
is interpreted as the $k$-trace generating functions of certain extended two-rowed plane partitions (see Figure \ref{FullGF}). This allows us to calculate many enumerative generating functions related to extended two-rowed plane partitions.
Consequently, we derive the number of linear extensions for these posets.

This paper is organized as follows.
In Section 2, we introduce the basic knowledge related to $P$-partitions, order polynomials, and linear extensions.
Section 3 delves into the combinatorial interpretation of the operators $\varphi_{z}$ and $\psi_{z}$, providing a proof of Theorem \ref{FullTwoPlane} along with its combinatorial interpretation.
Sections 4 and 5 are applications of Theorem \ref{FullTwoPlane}
to the determination of the enumerative generating functions and linear extensions.
Section 4 focuses on skew plane partitions. The corresponding order polynomials are known to have determinant formulas.
Section 5 considers three new types of posets:
Two-rowed plane partitions with diamond posets, an extension of the $V$-poset, and an extension of the ladder poset.
Throughout this paper, $\mathbb{Z}$, $\mathbb{N}$, and $\mathbb{P}$ denote the set of all integers, non-negative integers, and positive integers, respectively.

\section{$P$-Partitions}
We assume that the readers possess basic knowledge of posets (see, e.g., \cite[Section 3]{Stanley-Vol-1}).
Let $(P,\preceq)$ be a finite \emph{partially ordered set} (\emph{poset} for short).
We say that $t$ \emph{covers} $s$ if $s\prec t$ and there is no $c\in P$ such that $s\prec c\prec t$.
The \emph{Hasse diagram} of a finite poset $P$ is a graph where vertices are the elements of $P$, edges represent cover relations, and  if $s\prec t$, then $t$ is drawn above $s$ in the diagram.

We now introduce the concept of $P$-partitions, which were initially studied by MacMahon \cite[Subsections 439,441]{MacMahon19}.
In 1970, Knuth \cite{Knuth70} provided clarification on MacMahon's work within the theory of $P$-partitions.
However, the first comprehensive development of $P$-partitions was presented by Stanley in \cite{Stanley-Order-parti71,Stanley-Order-parti72}.

A $P$-partition is a map $\sigma: P\rightarrow \mathbb{N}$ that satisfies the following condition:
If $s\prec t$ in $P$, then $\sigma(s)\geq \sigma(t)$. In other words, $\sigma$ is \emph{order-reversing}.
If the sum $\sum_{s\in P}\sigma(s)$ equals $n$, denoted as $|\sigma|=n$, then $\sigma$ is referred to as a \emph{$P$-partition of $n$}.
It is convenient to represent $s\prec t$ with an arrow $s\to t$. Consequently, a $P$-partition is a vertex labeling $\sigma$ such that $s\to t$ implies $\sigma(s)\geq \sigma(t)$.

We denote the set of all $P$-partitions on the poset $P$ as $\pi(P)$.
The fundamental generating function associated with $\pi(P)$ is given by
$$F_P(x_1,x_2,\ldots,x_p)=\sum_{\sigma\in \pi(P)}x_1^{\sigma(1)}x_2^{\sigma(2)}\cdots x_p^{\sigma(p)}.$$
By setting $x_i=q$ for all $i$, we obtain the \emph{enumerative generating function} of $\pi(P)$:
\begin{align*}
\mathrm{PF}(\pi(P))=F_P(q,q,\ldots,q)=\sum_{\sigma\in \pi(P)} q^{|\sigma|}.
\end{align*}

The set $[p]$ with its usual order forms a $p$-element chain, i.e., a poset whose elements are totally ordered. This poset is denoted $\mathbf{p}$.
Let $\#P=p$. An order-preserving bijection $\tau: P\rightarrow \mathbf{p}$ is called a \emph{linear extension} of $P$.
The number of linear extensions of $P$ is denoted $e(P)$.

Now, suppose $(P,\preceq)$ is a finite poset on $[p]:=\{1,2,\ldots,p\}$.
We may identify a linear extension $\tau: P\rightarrow \mathbf{p}$ with the permutation $\tau^{-1}(1),\ldots,\tau^{-1}(p)$.
The set of all $e(P)$ permutations of $[p]$ obtained in this manner is denoted $\mathcal{L}(P)$ and is referred to as the \emph{Jordan-H\"older set} of $P$.
Let $\mu=\mu_1\mu_2\cdots \mu_p$ be a permutation. For $1\leq i\leq p-1$, if $\mu_i>\mu_{i+1}$, then $i$ is a \emph{descent} of $\mu$. Define the \emph{descent set} $D_{\mu}$ of $\mu$ by
$$D_{\mu}=\{ i : \mu_{i}> \mu_{i+1}\} \subseteq [p-1].$$
Let $d(\mu)$ be the number of descents of permutation $\mu$, i.e., $d(\mu)=\# D_{\mu}$.

\begin{thm}{\em \cite[Theorem 3.15.5]{Stanley-Vol-1}}
Let $(P,\preceq)$ be a finite poset on $[p]:=\{1,2,\ldots,p\}$. Then
\begin{align*}
F_P(x_1,x_2,\ldots, x_p)=\sum_{\mu\in \mathcal{L}(P)}\frac{\prod_{j\in D_{\mu}}x_{\mu_1}x_{\mu_2}\cdots x_{\mu_j}}{\prod_{i=1}^p (1-x_{\mu_1}x_{\mu_2}\cdots x_{\mu_i})}.
\end{align*}
\end{thm}

Define the \emph{major index} $\mathrm{maj}(\mu)$ of $\mu$ by $\mathrm{maj}(\mu)=\sum_{j\in D_{\mu}} j$.
\begin{thm}{\em \cite[Theorem 3.15.7]{Stanley-Vol-1}}
Let $(P,\preceq)$ be a finite poset on $[p]:=\{1,2,\ldots,p\}$. Then
\begin{align*}
\mathrm{PF}(\pi(P))=\frac{\sum_{\mu\in \mathcal{L}(P)} q^{\mathrm{maj}(\mu)}}{(1-q)(1-q^2)\cdots (1-q^p)}.
\end{align*}
\end{thm}

Now, let's consider an alternative method for counting $P$-partitions.
For $m\in \mathbb{P}$, let $\aleph(P,m)$ denote the number of $P$-partitions $P\rightarrow \mathbf{m} \uplus \{0\}$.

Let $\Omega_{P}(m)$ denote the number of order-preserving maps $\rho: P\rightarrow \mathbf{m}$.
In fact, $\Omega_{P}(m)$ is a polynomial function \cite[P. 334]{Stanley-Vol-1} of $m$ of degree $p$ with leading coefficient $\frac{e(P)}{p!}$.
Thus $\Omega_{P}(m)$ is called the \emph{order polynomial} of $P$.
By replacing $\rho(t)$ with $m+1-\rho(t)$, we see that $\Omega_P(m)$ is also the number of order-reversing maps $P\rightarrow \mathbf{m}$, i.e., the number of $P$-partitions $P\rightarrow \mathbf{m}$.
Therefore, we have $\aleph(P,m)=\Omega_P(m+1)$ for any $m\geq 0$.

\begin{thm}{\em \cite[Theorem 3.15.8]{Stanley-Vol-1}}
Let $(P,\preceq)$ be a finite poset on $[p]:=\{1,2,\ldots,p\}$. Then
\begin{align*}
\sum_{m\geq 0}\aleph(P,m)x^m=\sum_{m\geq 0}\Omega_{P}(m+1)x^m=\frac{\sum_{\mu \in \mathcal{L}(P)}x^{d(\mu)}}{(1-x)^{p+1}}.
\end{align*}
\end{thm}

Finally, we conclude this section by providing two formulas for the number $e(P)$ of linear extensions. They are valuable since Brightwell and Winller \cite{Brightwell91} showed that counting $e(P)$ of a finite poset is $\#\mathbf{P}$-complete.

\begin{thm}{\em \cite[Section 3]{Stanley-Vol-1}}
Let $(P,\preceq)$ be a finite poset on $[p]:=\{1,2,\ldots,p\}$. Then
\begin{align}
e(P)&=p!\cdot \bigg(\Big((1-q)^p\cdot\mathrm{PF}(\pi(P))\Big)\Big|_{q=1}\bigg).\label{linearExtension} \\
e(P)&=p!\cdot \big([m^p]\Omega_{P}(m)\big)=p!\cdot \big([m^p]\aleph(P,m)\big), \label{linearExtension-Second}
\end{align}
where $[m^i]f(m)$ denotes the coefficient of $m^i$ in a polynomial $f(m)$.
\end{thm}

\section{An Extension of Two-Rowed Plane Partitions}

Now suppose $P$ is a poset containing the segment $A\to B$ (see the left side of Figure \ref{map1}). Consider the generating function
\begin{align*}
F(x,y;P)=\sum_{\sigma\in \pi(P)} x^{\sigma(A)} y^{\sigma(B)} q^{|\sigma|-\sigma(A)-\sigma(B)}= \sum_{i\geq j\geq 0}a_{ij}x^iy^j.
\end{align*}
This generating function keeps track of $\sigma(A)$ and $\sigma(B)$. In particular,
the coefficient $a_{ij}$ is the weighted counting series (in $q$) for $P$-partitions $\sigma$ of $P$ satisfying $\sigma(A)=i$, $\sigma(B)=j$.
By definition $i=\sigma(A)\geq \sigma(B)=j$. Our main concern here is about $x$ and $y$.

We construct two posets $\widetilde{P}_1$ and $\widetilde{P}_2$ to explain the action of the two operators $\varphi_{z}$ and $\psi_{z}$ defined in \eqref{Varphi-map} and \eqref{Psimap} on $F(x,y;P)$, respectively.

The poset $\widetilde{P}_1$ is obtained from $P$ by adding two new vertices $C$ and $D$, with the arrows $C\to A\to  D \to B$. See Figure \ref{map1}, where
a $P$-partition $\sigma$ of $P$ is extended as $\sigma(A)=i$, $\sigma(B)=j$, $\sigma(C)=m$, $\sigma(D)=\ell$ for $\widetilde{P}_1$. Note that
the original arrow $A\to B$ is omitted due to the transitivity of the partial order.
\begin{figure}[htp]
\centering
\includegraphics[width=12cm,height=4cm]{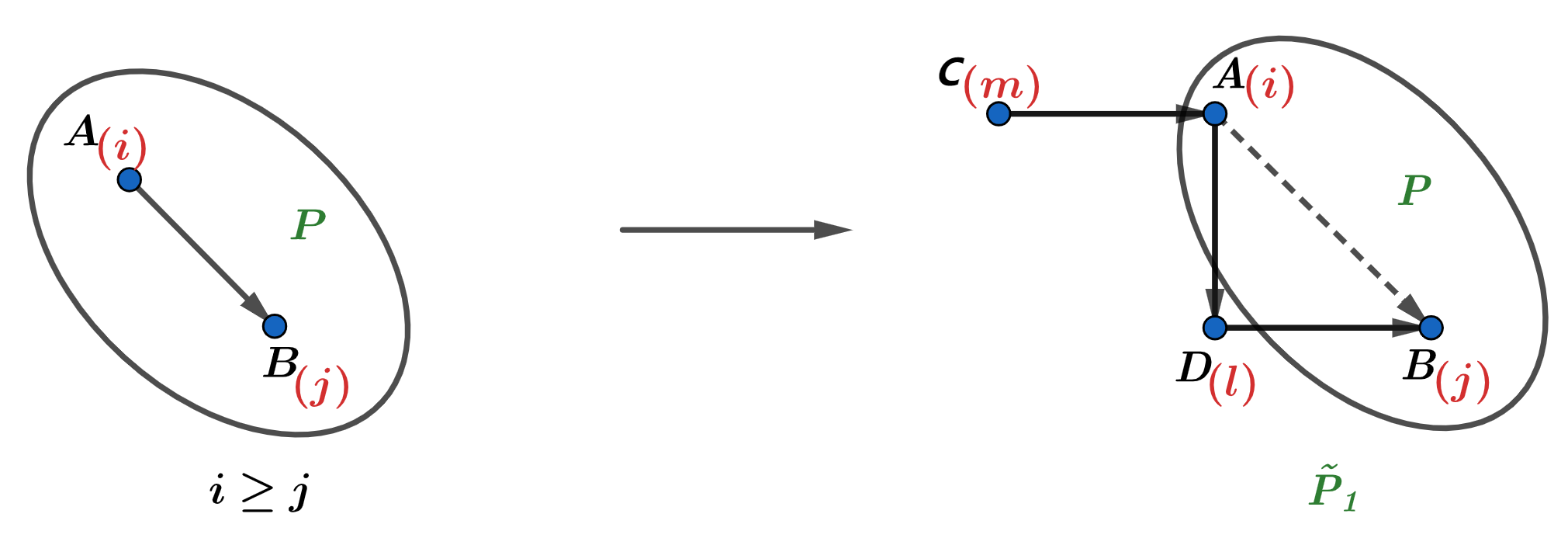}
\caption{The poset $\widetilde{P}_1$.}
\label{map1}
\end{figure}
In this new poset, we keep track of $\sigma(C)$, $\sigma(D)$, and the trace $\sigma(A)+\sigma(B)$ by letting
$$ F(x,y;z; \widetilde{P}_1) = \sum_{\sigma \in \pi(\widetilde{P}_1)} x^{\sigma(C)}y^{\sigma(D)} z^{\sigma(A)+\sigma(B)} q^{|\sigma|-\sigma(C)-\sigma(D)-(\sigma(A)+\sigma(B))}.$$
Then the following lemma gives a combinatorial interpretation of the operator $\varphi_{z}$.
\begin{lem}\label{Map1-Varphi}
Follow the notation as above. We have
$$\varphi_{z}\odot F(x,y;P)=\sum_{m\geq i\geq \ell\geq j\geq 0}a_{ij}x^m y^{\ell} z^{i+j}
=F(x,y;z;\widetilde{P}_1).$$
\end{lem}
\begin{proof}
By \eqref{Varphi-map}, we have
\begin{align*}
\varphi_{z}\odot F(x,y;P)& =\frac{\left(\sum_{i\geq j\geq 0}a_{ij}x^iy^jz^{i+j}-\sum_{i\geq j\geq 0}a_{ij}x^iy^{i+1}z^{i+j}\right)}{1-y}\sum_{k\geq 0}x^{k}\\
&=\sum_{i\geq j\geq 0, k\geq 0}a_{ij}x^{i+k}z^{i+j}\sum_{\ell=j}^{i}y^\ell.
\end{align*}
This means that for each pair $i\geq j$, a $P$-partition $\sigma$ of $P$ with $\sigma(A,B)=(i,j)$ generates
$P$-partitions $\sigma'$ of $\widetilde{P}_1$ with $\sigma'(A,B,C,D)=(i,j,i+k,\ell)$ with the desired conditions $i+k\geq i \geq \ell \geq j$.
\end{proof}

Similarly, the poset $\widetilde{P}_2$ is obtained from $P$ by adding the new vertex $E$ with the arrows $A\to  E \to B$. See Figure \ref{map2},
where
a $P$-partition $\sigma$ of $P$ is extended as $\sigma(A)=i$, $\sigma(B)=j$, $\sigma(E)=\ell$ for $\widetilde{P}_2$. Again
the original arrow $A\to B$ is omitted similarly.
\begin{figure}[htp]
\centering
\includegraphics[width=12cm,height=4cm]{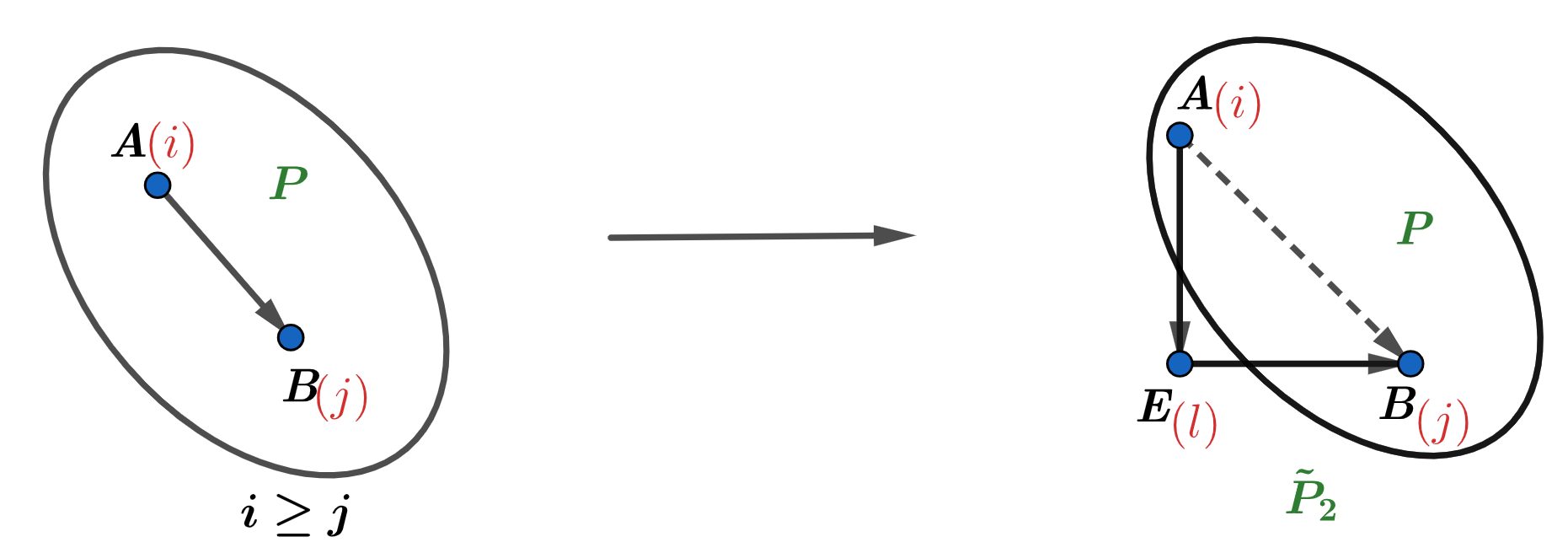}
\caption{The poset $\widetilde{P}_2$.}
\label{map2}
\end{figure}

In the poset $\widetilde{P}_2$, we keep track of $\sigma(E)$ and the trace $\sigma(A)+\sigma(B)$ by letting
$$F(x;z; \widetilde{P}_2) = \sum_{\sigma \in \pi(\widetilde{P}_2)} x^{\sigma(E)}z^{\sigma(A)+\sigma(B)} q^{|\sigma|-\sigma(E)-(\sigma(A)+\sigma(B))}.$$
Similar to Lemma \ref{Map1-Varphi}, we can obtain the following results.
\begin{lem}\label{Map2-Psi}
Follow the notation as above. We have
$$\psi_{z}\odot F(x,y;P)=\sum_{i\geq \ell\geq j\geq 0}a_{ij}x^{\ell} z^{i+j}
=F(x;z; \widetilde{P}_2).$$
\end{lem}

Our main result, Theorem \ref{FullTwoPlane}, is derived by composing these operators. 
This result enables us to obtain many generating functions associated with two-rowed plane partitions.
For the reader's convenience, we restate the theorem as follows. 
\begin{thm}[also Theorem \ref{FullTwoPlane}]
The operators $\varphi_{z}$, $\psi_{z}$ are defined in \eqref{Varphi-map} and \eqref{Psimap}.
Let $F(x,y;P)=\sum_{i\geq j\geq 0}a_{ij}x^iy^j$. Let $n\geq 2$. Then we have
\begin{align*}
&\psi_{z_n}\odot \varphi_{z_{n-1}}\odot \varphi_{z_{n-2}}\odot \cdots\odot \varphi_{z_1}\odot F(x,y;P)
\\=&\frac{F(z_1 \cdots z_n, xz_1\cdots z_n;P)-xF(xz_1\cdots z_n,z_1\cdots z_n;P)}{(1-x)(1-xz_n)(1-xz_{n-1} z_n)\cdots(1-xz_2\cdots z_n)(1-z_n)(1-z_{n-1}z_n)\cdots(1-z_2\cdots z_n)}.
\end{align*}
\end{thm}
\begin{proof}
We prove the theorem by induction on $n$. By \eqref{Varphi-map} and \eqref{Psimap}, we know that
\begin{small}
\begin{align*}
\varphi_{z}\odot F(x,y;P)=\frac{F(xz,yz;P)-yF(xyz,z;P)}{(1-x)(1-y)}, \ \ \
\psi_{z}\odot F(x,y;P)=\frac{F(z,xz;P)-xF(xz,z;P)}{1-x}.
\end{align*}
\end{small}
In the base case when $n=2$, we have
\begin{align*}
&\psi_{z_2}\odot \varphi_{z_1}\odot  F(x,y;P)
\\=&\psi_{z_2}\odot \frac{F(xz_1,yz_1;P)-yF(xyz_1,z_1;P)}{(1-x)(1-y)}
\\=&\frac{1}{1-x} \left(\frac{F(z_1z_2,xz_1z_2;P)-xz_2F(xz_1z_2^2,z_1;P)}{(1-z_2)(1-xz_2)}-x\cdot\frac{F(xz_1z_2,z_1z_2;P)-z_2F(xz_1z_2^2,z_1;P)}{(1-z_2)(1-xz_2)}\right)
\\=&\frac{F(z_1z_2,xz_1z_2;P)-xF(xz_1z_2,z_1z_2;P)}{(1-x)(1-z_2)(1-xz_2)}.
\end{align*}
Assuming the theorem holds for $n-1$. By simple variable substitution, we obtain
\begin{align*}
& \psi_{z_n}\odot \varphi_{z_{n-1}}\odot \cdots\odot \varphi_{z_2}\odot F(x,y;P)
\\=& \frac{F(z_2 \cdots z_n, xz_2\cdots z_n;P)-xF(xz_2\cdots z_n,z_2\cdots z_n;P)}{(1-x)(1-xz_n)(1-xz_{n-1} z_n)\cdots(1-xz_3\cdots z_n)(1-z_n)(1-z_{n-1}z_n)\cdots(1-xz_3\cdots z_n)}.
\end{align*}
Therefore, we get
\begin{align*}
&\psi_{z_n}\odot \varphi_{z_{n-1}}\odot \cdots\odot \varphi_{z_1} \odot F(x,y;P)
\\=&\psi_{z_n}\odot \varphi_{z_{n-1}}\odot \cdots\odot \varphi_{z_2} \odot \frac{F(xz_1,yz_1;P)-yF(xyz_1,z_1;P)}{(1-x)(1-y)}
\\=&\frac{\frac{F(z_1 \cdots z_n, xz_1\cdots z_n;P)-xz_2\cdots z_nF(xz_1z_2^2\cdots z_n^2,z_1;P)}{(1-xz_2\cdots z_n)(1-z_2\cdots z_n)}-x\cdot\frac{F(xz_1 \cdots z_n, z_1\cdots z_n;P)-z_2\cdots z_nF(xz_1z_2^2\cdots z_n^2,z_1;P)}{(1-xz_2\cdots z_n)(1-z_2\cdots z_n)}}{(1-x)(1-xz_n)(1-xz_{n-1} z_n)\cdots(1-xz_3\cdots z_n)(1-z_n)(1-z_{n-1}z_n)\cdots(1-z_3\cdots z_n)}
\\=&\frac{F(z_1 \cdots z_n, xz_1\cdots z_n;P)-x\cdot F(xz_1\cdots z_n,z_1\cdots z_n;P)}{(1-x)(1-xz_n)(1-xz_{n-1} z_n)\cdots(1-xz_2\cdots z_n)(1-z_n)(1-z_{n-1}z_n)\cdots(1-z_2\cdots z_n)}.
\end{align*}
This completes the proof.
\end{proof}

We consider the combinatorial explanation of $\psi_{z_n}\odot \varphi_{z_{n-1}}\odot \varphi_{z_{n-2}}\odot \cdots\odot \varphi_{z_1}\odot F(x,y;P)$.
Based on the posets $\widetilde{P}_1$ and $\widetilde{P}_2$ in Figures \ref{map1} and \ref{map2}, we can construct the poset $\overline{P}$ in Figure \ref{FullGF}.
\begin{figure}[htp]
\centering
\includegraphics[width=16cm,height=4cm]{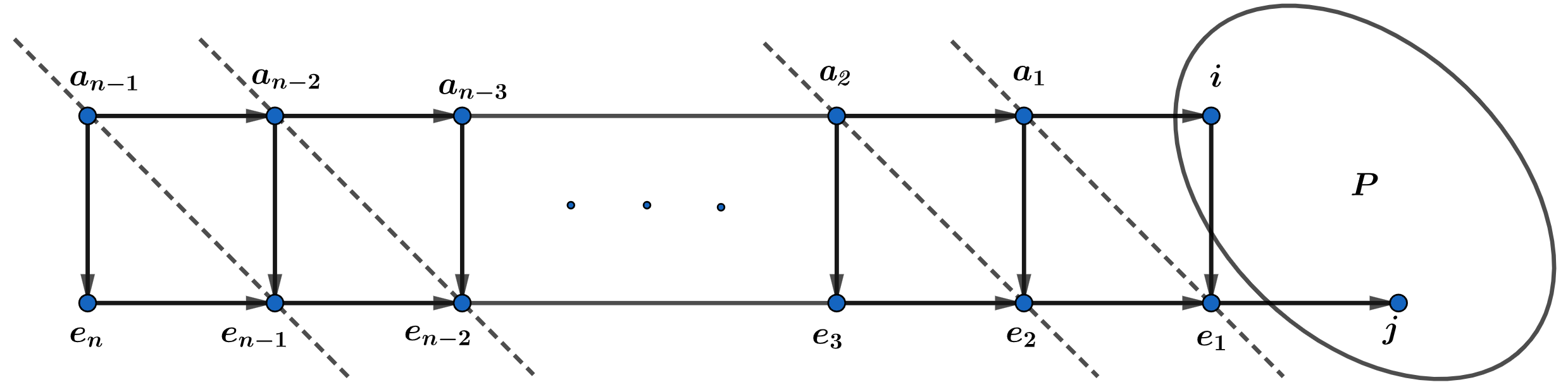}
\caption{The poset $\overline{P}$.}
\label{FullGF}
\end{figure}

The partial conditions of the $P$-partition of $\overline{P}$ can be described as:
We label the vertices by nonnegative integers in poset $\overline{P}$ (see Figure \ref{FullGF}) such that the following conditions (\text{Conditions\ } \textbf{P}) hold:
\begin{align*}
\text{Conditions\ } \textbf{P}=\left\{
             \begin{array}{ll}
             a_{n-1}\geq a_{n-2}\geq \cdots \geq a_1\geq i\geq 0; \\
             e_{n}\geq e_{n-1}\geq \cdots \geq e_1\geq j\geq 0; \\
             a_k\geq e_{k+1}\ \ \text{for}\ \ 1\leq k\leq n-1;\\
             i\geq e_1.
             \end{array}
\right.
\end{align*}

We consider $\psi_{z_n}\odot \varphi_{z_{n-1}}\odot \varphi_{z_{n-2}}\odot \cdots\odot \varphi_{z_1}\odot F(x,y;P)$ as a generating function related to $\overline{P}$.
\begin{thm}\label{TwoPlaneQq}
Follow the notation as above. We have
\begin{align*}
&\psi_{z_n}\odot \varphi_{z_{n-1}}\odot \varphi_{z_{n-2}}\odot \cdots\odot \varphi_{z_1}\odot F(x,y;P)
\\=&\sum_{\text{Conditions\ } \textbf{P}}a_{ij}x^{e_n}z_n^{a_{n-1}+e_{n-1}}z_{n-1}^{a_{n-2}+e_{n-2}}\cdots z_2^{a_1+e_1}z_1^{i+j}.
\end{align*}
In particular, we have
$$\psi_{z_n}\odot  \varphi_{z_{n-1}}\odot \varphi_{z_{n-2}}\odot \cdots\odot \varphi_{z_1}\odot F(x,y;P)\big|_{x=z_1=\cdots=z_n=q}
=\sum_{\sigma \in \pi(\overline{P})}q^{|\sigma|}.$$
\end{thm}
\begin{proof}
By Lemma \ref{Map1-Varphi} and \ref{Map2-Psi}, it is now obvious that the theorem holds.
\end{proof}

In fact, the $P$-partition of the poset $\overline{P}$ is an extension of the two-rowed plane partition.
For an extended plane partition $\Pi \in\pi(\overline{P})$, we have the
$\mathrm{tr}_{-1}(\Pi)=e_n$, $\mathrm{tr}_{k}(\Pi)=a_{n-1-k}+e_{n-1-k}$ for $0\leq k\leq n-2$.
We will gain a deeper understanding of this fact in the following corollaries.

Before we proceed again, let's first introduce a notation.
The \emph{$q$-shifted factorial} is defined as
$$(a;q)_0=1;\ \ \ \ (a;q)_n=(1-a)(1-aq)\cdots(1-aq^{n-1}),$$
where $n$ is any positive integer.

\begin{cor}\label{Main-Corollary2}
Follow the notation in Theorems \ref{FullTwoPlane} and \ref{TwoPlaneQq}.
If $z_1=z_2=\cdots =z_n=q$, then we have
\begin{align}
\mathrm{TF}(\pi(\overline{P})):&=\psi_{z_n}\odot \varphi_{z_{n-1}}\odot \varphi_{z_{n-2}}\odot \cdots\odot \varphi_{z_1}\odot F(x,y;P)\big|_{z_1=\cdots=z_n=q}\nonumber
\\&=\frac{F(q^n,xq^{n};P)-xF(xq^{n},q^n;P)}{(x;q)_n (q;q)_{n-1}}.\label{FormuTF}
\end{align}
If $x=z_1=z_2=\cdots =z_n=q$, then the enumerative generating function of the set $\pi(\overline{P})$ is given by
\begin{align}
\mathrm{PF}(\pi(\overline{P}))&=\psi_{z_n}\odot \varphi_{z_{n-1}}\odot \varphi_{z_{n-2}}\odot \cdots\odot \varphi_{z_1}\odot F(x,y;P)\big|_{x=z_1=\cdots=z_n=q}\nonumber
\\&=\frac{F(q^n,q^{n+1};P)-qF(q^{n+1},q^n;P)}{(q;q)_n (q;q)_{n-1}}.\label{FormuPF}
\end{align}
\end{cor}

\section{Skew Plane Partitions}

Let $\lambda$ and $\mu$ be two partitions where $\mu\subseteq \lambda$ (i.e., $\mu_i\leq \lambda_i$ for all $i$).
Define the \emph{length} $\ell(\lambda)$ as the number of non-zero parts.
Define a \emph{skew plane partition} of skew shape $\lambda/\mu$ to be an array $T=(T_{ij})$ of nonnegative integers of shape $\lambda/\mu$ (i.e., $1\leq i\leq \ell(\lambda)$, $\mu_i <j\leq \lambda_i$) that is weakly decreasing in both rows and columns.
For instance, the following array
\begin{align*}
\begin{matrix}
& & & 7& 4 \\
& & 6& 5& 5\\
& & 6& 3& \\
&9 & 0& & \\
\end{matrix}
\end{align*}
represents a skew plane partition of the skew shape \(\lambda/\mu = (4,4,3,2)/(2,1,1)\).

Skew plane partitions can be regarded as \(P\)-partitions of a poset, denoted as \(P_{\lambda/\mu}\). Let \(m \in \mathbb{P}\). In Section 2, \(\aleph(P,m)\) denotes the number of \(P\)-partitions  \(P \rightarrow \mathbf{m} \uplus \{0\}\).
Now we set $m\in \mathbb{N}$ without leading to ambiguity.
Therefore, $\aleph(P_{\lambda/\mu},m)$ is the number of skew plane partitions of shape $\lambda/\mu$ such that the largest part is less than or equal to $m$.

A determinant formula for $\aleph(P_{\lambda/\mu},m)$ was first given by Kreweras \cite[Section 2.3.7]{G.Kreweras} (also see \cite[Section 3; Exercise 149]{Stanley-Vol-1}), which is expressed as follows:
\begin{align}\label{P-determinant}
\aleph(P_{\lambda/\mu},m)=\det_{1\leq i,j\leq k}\left(\binom{m+\lambda_i-\mu_j}{m+i-j}\right),
\end{align}
where $k=\ell(\lambda)$, and $\mu_j=0$ for $\ell(\mu)<j\leq \ell(\lambda)$. In the specific case where \(\lambda/\mu = (M-d, M-2d, \ldots, M-kd)/\emptyset\), the determinant can be explicitly computed, as detailed in \cite[Exercise 7.101(b)]{RP.Stanley2024}.

To ensure completeness, we provide a proof for the following proposition:
\begin{prop}
The number of linear extensions of the poset \(P_{\lambda/\mu}\) is given by the formula:
\begin{align}\label{Linear-Exten-Skew-Deter}
e(P_{\lambda/\mu})=(|\lambda|-|\mu|)!\cdot \det_{1\leq i,j\leq k}\left(\frac{1}{(\lambda_i-\mu_j-i+j)!}\right),
\end{align}
where $k=\ell(\lambda)$, and $\mu_j=0$ for $\ell(\mu)<j\leq \ell(\lambda)$.
\end{prop}
\begin{proof}
By \eqref{linearExtension-Second} and \eqref{P-determinant}, we have
\begin{align*}
e(P_{\lambda/\mu})&=(|\lambda|-|\mu|)!\cdot \Big([m^{|\lambda|-|\mu|}]\aleph(P_{\lambda/\mu},m)\Big)
\\&=(|\lambda|-|\mu|)!\cdot [m^{|\lambda|-|\mu|}] \det_{1\leq i,j\leq k}\left(\binom{m+\lambda_i-\mu_j}{\lambda_i-\mu_j-i+j}\right)
\\&=(|\lambda|-|\mu|)!\cdot \det_{1\leq i,j\leq k}\left(\frac{1}{(\lambda_i-\mu_j-i+j)!}\right).
\end{align*}
This completes the proof.
\end{proof}

In fact, Equation \eqref{Linear-Exten-Skew-Deter} is also the formula to count the number of standard Young tableaux of shape $\lambda/\mu$ \cite[Corollary 7.16.3]{RP.Stanley2024}.

In this section, we utilize Corollary \ref{Main-Corollary2} to calculate the enumerative generating functions $\mathrm{PF}(\pi(P_{\lambda/\mu}))$ for certain skew plane partitions $P_{\lambda/\mu}$.
Furthermore, we employ Equation \eqref{linearExtension} to obtain their linear extensions.
This is consistent with the formula derived from Equation \eqref{Linear-Exten-Skew-Deter}.

\begin{figure}[htp]
\centering
\includegraphics[width=15cm,height=3cm]{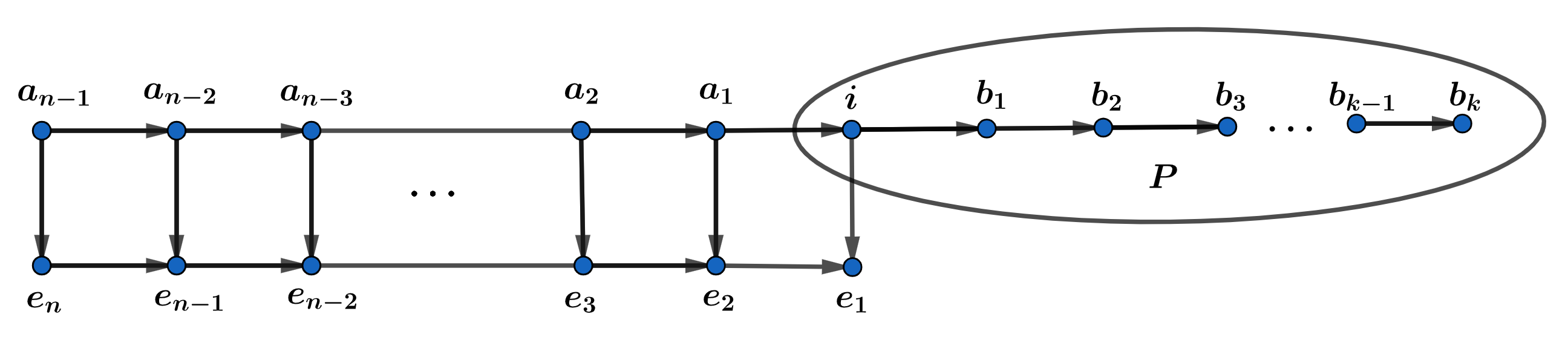}
\caption{The poset $P_0=P_{(n+k,n)/\emptyset}$.}
\label{LineGraph}
\end{figure}

When $P$ is the poset in Figure \ref{LineGraph}, we have the generating function
$$F(x,y;P)=\sum_{i\geq b_1\geq b_2\geq \cdots \geq b_k\geq 0}x^iy^0q^{b_1+b_2+\cdots +b_k}=\frac{1}{(x;q)_{k+1}}.$$

\begin{cor}\label{linegraphEGF}
Let $F(x,y;P)=\frac{1}{(x;q)_{k+1}}$.
Then the enumerative generating function of the set $\pi(P_0)$ is given by
\begin{align*}
\mathrm{PF}(\pi(P_0))=\frac{1-q+q^{n+1}-q^{n+k+1}}{(q;q)_{n+k+1}(q;q)_{n}}.
\end{align*}
Furthermore, the number of linear extensions of $P_0$ is
$e(P_0)=\frac{(2n+k)!(k+1)}{n!(n+k+1)!}$.
\end{cor}
\begin{proof}
By \eqref{FormuPF}, we have
\begin{align*}
\mathrm{PF}(\pi(P_0))=\frac{\frac{1}{(q^n;q)_{k+1}}-q\cdot \frac{1}{(q^{n+1};q)_{k+1}}}{(q;q)_{n}(q;q)_{n-1}}=\frac{1-q+q^{n+1}-q^{n+k+1}}{(q;q)_{n+k+1}(q;q)_{n}}.
\end{align*}
By \eqref{linearExtension}, we obtain the $e(P_0)$.
This completes the proof.
\end{proof}

In Corollary \ref{linegraphEGF}, if $k=0$, then we get the generating function for two-rowed plane partitions of length $n$.
This result was obtained by Andrews in \cite[Theorem 6.1]{Andrews2}. Andrews used MacMahon's partition analysis to obtain this result. Our process may be more straightforward.

For $\aleph(P_{(\lambda,\mu)/\emptyset},m)$, we may not anticipate an exact generation function formula.
\begin{exa}
With the help of \texttt{Maple}, we can obtain the following generation functions
\begin{small}
\begin{align*}
\sum_{m=0}^{\infty}\aleph(P_{(k+1,1)/\emptyset},m)x^m&=\frac{1+kx}{(1-x)^{k+3}};
\\ \sum_{m=0}^{\infty}\aleph(P_{(k+2,2)/\emptyset},m)x^m&=\frac{1+(2k+1)x+\frac{k(k+1)}{2}x^2}{(1-x)^{k+5}};
\\ \sum_{m=0}^{\infty}\aleph(P_{(k+3,3)/\emptyset},m)x^m&=\frac{1+3(k+1)x+\frac{(3k+1)(k+2)}{2}x^2
+\frac{k(k+1)(k+2)}{6}x^3}{(1-x)^{k+7}};
\\ \sum_{m=0}^{\infty}\aleph(P_{(k+4,4)/\emptyset},m)x^m&=\frac{1+(4k+6)x+(3k^2+11k+6)x^2+\frac{4k^3+21k^2+29k+6}{6}x^3
+\frac{k(k+1)(k+2)(k+3)}{24}x^4}{(1-x)^{k+9}}.
\end{align*}
\end{small}
\end{exa}

By \eqref{FormuTF}, if $F(x,y;P)=\frac{1}{1-x}$ (i.e., $k=0$), then we have
\begin{align}\label{SquareTF}
\mathrm{TF}(\pi(P_0))=\frac{1}{(xq;q)_n(q;q)_n}.
\end{align}

\begin{figure}[htp]
\centering
\includegraphics[width=13cm,height=4cm]{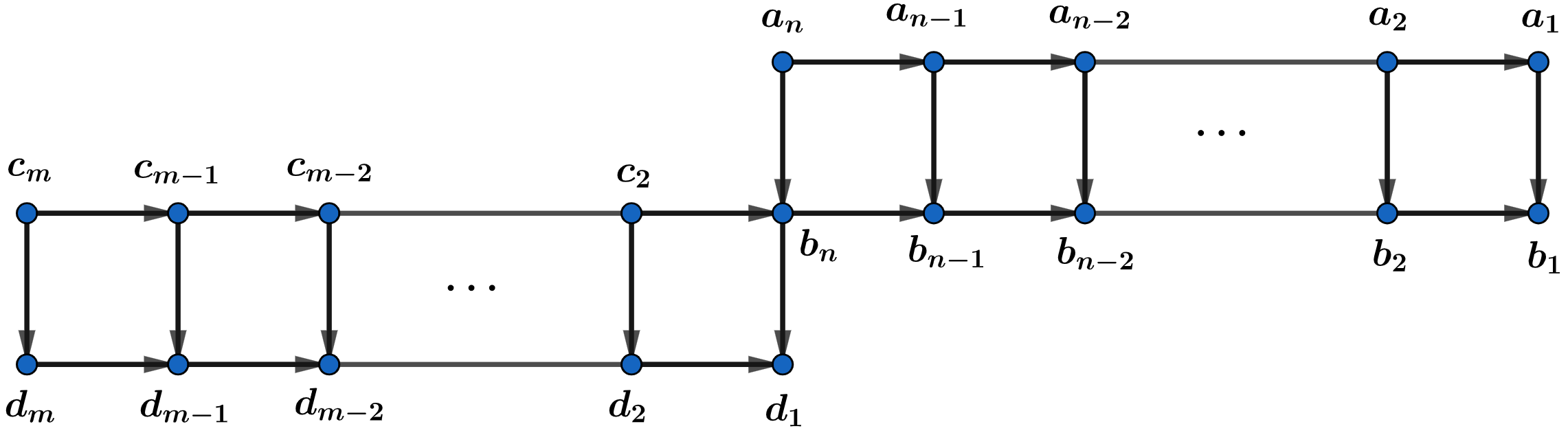}
\caption{The poset $P_1=P_{(n+m-1,n+m-1,m)/(m-1)}$.}
\label{PlaneP2}
\end{figure}

Now we consider the enumerative generating function of the set $\pi(P_1)$ in Figure \ref{PlaneP2}. This implies that $P$-partitions of the poset $P_1$ satisfy the following condition:
\begin{align*}
\left\{
             \begin{array}{ll}
             c_{m}\geq c_{m-1}\geq \cdots \geq c_2\geq b_n\geq b_{n-1}\geq \cdots \geq b_1\geq 0; \\
             d_{m}\geq d_{m-1}\geq \cdots \geq d_1\geq 0;\ \ a_{n}\geq a_{n-1}\geq \cdots \geq a_1; \\
             c_k\geq d_{k}\ \ \text{for}\ \ 2\leq k\leq m;\ \ b_n\geq d_1; \\
             a_k\geq b_{k}\ \ \text{for}\ \ 1\leq k\leq n.
             \end{array}
\right.
\end{align*}

\begin{cor}\label{corr1}
The enumerative generating function of the set $\pi(P_1)$ in Figure \ref{PlaneP2} is given by
$$\mathrm{PF}(\pi(P_1))=\frac{1-q-q^{m+n+1}+q^{m+2}}{(q;q)_{m-1}(q;q)_{m+n+1}(q;q)_n}.$$
Furthermore, the number of linear extensions of $P_1$ is
$e(P_1)=\frac{(2n+2m-1)!}{(m-1)!(n-1)!(m+n+1)!}$.
\end{cor}
\begin{proof}
According to \eqref{SquareTF}, let $F(x,y;P)=\frac{1}{(xq;q)_n(q;q)_n}$.
Then based on \eqref{FormuPF}, we can obtain
\begin{align*}
\mathrm{PF}(\pi(P_1))&=\frac{F(q^m,q^{m+1};P)-q\cdot F(q^{m+1};q^m;P)}{(q;q)_m (q;q)_{m-1}}
\\&=\frac{1}{(q;q)_m (q;q)_{m-1}}\left(\frac{1}{(q^{m+1};q)_{n}(q;q)_n}-q\cdot \frac{1}{(q^{m+2};q)_{n}(q;q)_n}\right)
\\&=\frac{1-q-q^{m+n+1}+q^{m+2}}{(q;q)_{m-1}(q;q)_{m+n+1}(q;q)_n}.
\end{align*}
By \eqref{linearExtension}, we obtain the $e(P_1)$.
The proof of the corollary is now complete.
\end{proof}

Similarly, we can obtain more generation functions. The following reasoning is analogous to that in Corollary \ref{corr1} and is left to the reader.

\begin{figure}[htp]
\centering
\includegraphics[width=16cm,height=5cm]{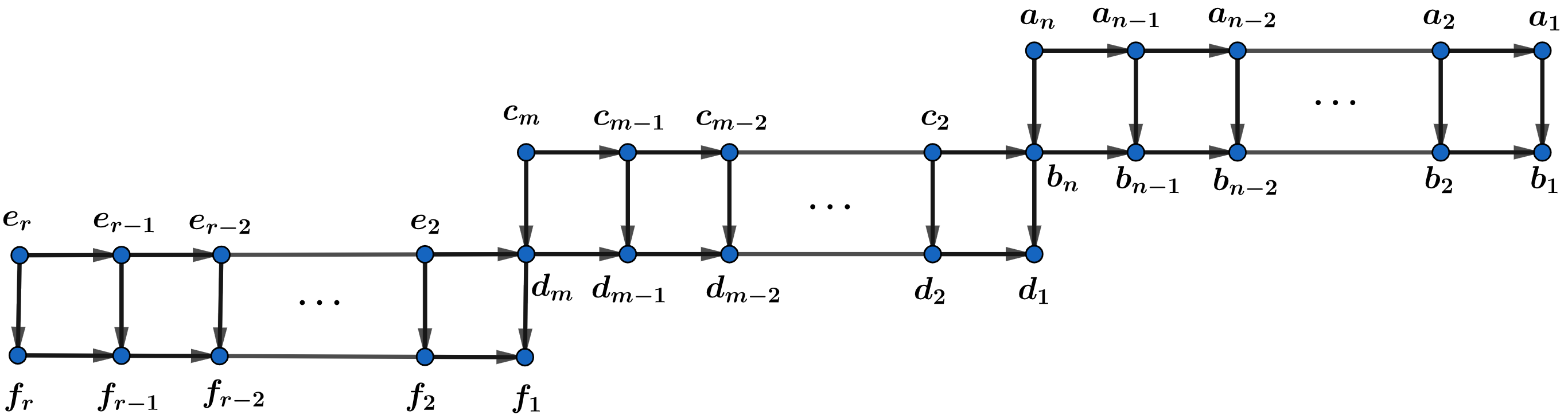}
\caption{The poset $P_2=P_{(n+m+r-2,n+m+r-2,m+r-1,r)/(m+r-2,r-1)}$.}
\label{PlaneP3}
\end{figure}

\begin{cor}\label{corr2}
The enumerative generating function of the set $\pi(P_2)$ in Figure \ref{PlaneP3} is given by
\begin{align*}
\mathrm{PF}(\pi(P_2))=&\frac{(1-q^m)(q^{r+1};q)_{m+n}(q^{m+r+1};q)_{n+1}(1-q^{m+r}-q+q^{r+1})}
{(q;q)_{r+m+n+1}(q;q)_{r+m+n}(q;q)_{m+n}(q;q)_n}
\\+&\frac{q^r(q^m;q)_{n+1}(q^{r+1};q)_{m+n}\big(q^2(1-q^r)(1-q^{m+r+1})-(1-q^{m+n+r+1})(1-q^{m+r})\big)}
{(q;q)_{r+m+n+1}(q;q)_{r+m+n}(q;q)_{m+n}(q;q)_n}.
\end{align*}
Furthermore, the number of linear extensions of $P_2$ is
$$e(P_2)=(2r+2n+2m-2)!\cdot\left(\frac{m^2}{n!r!(m+r)!(m+n)!}-\frac{m^2+m(n+r+1)+rn}{n!r!(m-1)!(m+n+r+1)!}\right).$$
\end{cor}

\begin{figure}[htp]
\centering
\includegraphics[width=16cm,height=3.8cm]{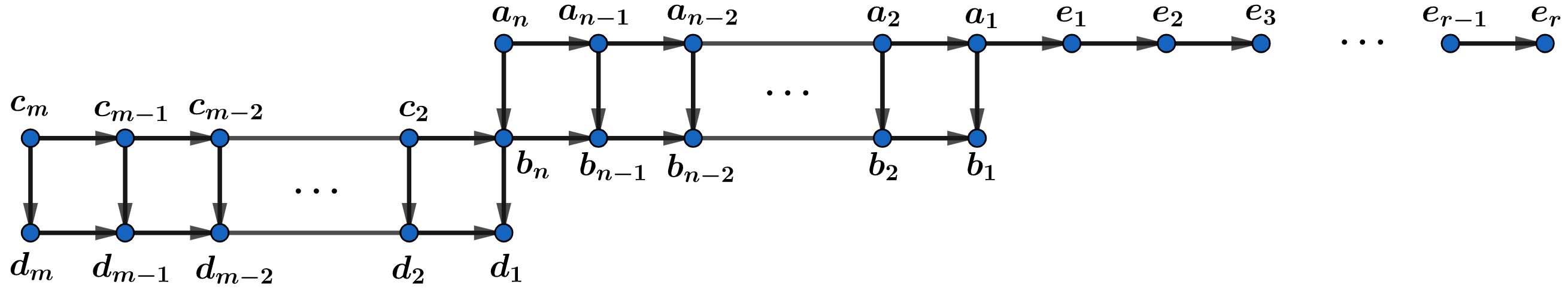}
\caption{The poset $P_3=P_{(n+r+m-1,n+m-1,m)/(m-1)}$.}
\label{PlaneP1}
\end{figure}

\begin{cor}\label{corr3}
The enumerative generating function of the set $\pi(P_3)$ in Figure \ref{PlaneP1} is given by
\begin{align*}
\mathrm{PF}(\pi(P_3))=\frac{(q^{m+n+1};q)_{r+1}(1-q-q^{m+n}+
q^{m+1})-q^m(q^n;q)_{r+1}(1-q^2+q^{m+2}-q^{n+r+m+1})}
{(q;q)_{n+r+m+1}(q;q)_m(q;q)_{n+r}}.
\end{align*}
Furthermore, the number of linear extensions of $P_3$ is
$$e(P_3)=(2n+2m+r-1)!\cdot\left(\frac{n}{m!(n+r)!(m+n)!}-\frac{n+r+1}{m!(n-1)!(m+n+r+1)!}\right).$$
\end{cor}

\section{New Formulas}

Similarly, we utilize Corollary \ref{Main-Corollary2} and Equation \eqref{linearExtension} to calculate the enumerative generating functions $\mathrm{PF}(\pi(P))$ and linear extensions for certain posets $P$.
These posets can be seen as the extension of two-rowed plane partitions.
In this section, the formula we have obtained is nontrivial.

\subsection{Two-Rowed Plane Partition Extension of the Diamond Posets}

We consider the diamond posets studied in \cite{Andrews8,Andrews10}.
In \cite{Andrews8}, Andrews obtained the following full generating function of the diamond poset in Figure \ref{Square1}:
\begin{align}
D(x_1,x_2,x_3,x_4)&=\sum_{u_1\geq u_2;\ u_1\geq u_3; \atop u_2\geq u_4;\ u_3\geq u_4\geq 0}x_1^{u_1}x_2^{u_2}x_3^{u_3}x_4^{u_4}\nonumber
\\&=\frac{1-x_1^2x_2x_3}{(1-x_1)(1-x_1x_2)(1-x_1x_3)(1-x_1x_2x_3)(1-x_1x_2x_3x_4)}.\label{DiamondDDXY}
\end{align}

\begin{figure}[htp]
\centering
\includegraphics[width=3cm,height=3cm]{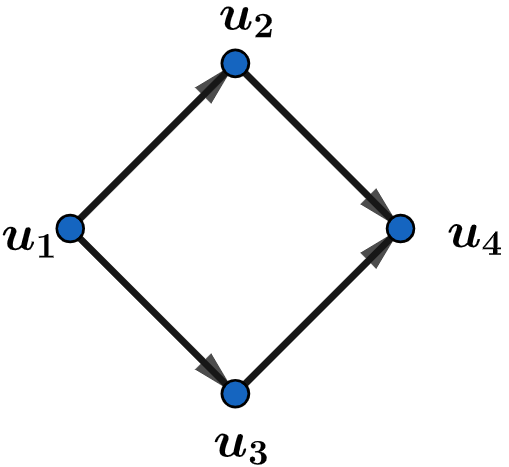}
\caption{The diamond poset.}
\label{Square1}
\end{figure}

\begin{figure}[htp]
\centering
\includegraphics[width=10cm,height=3.5cm]{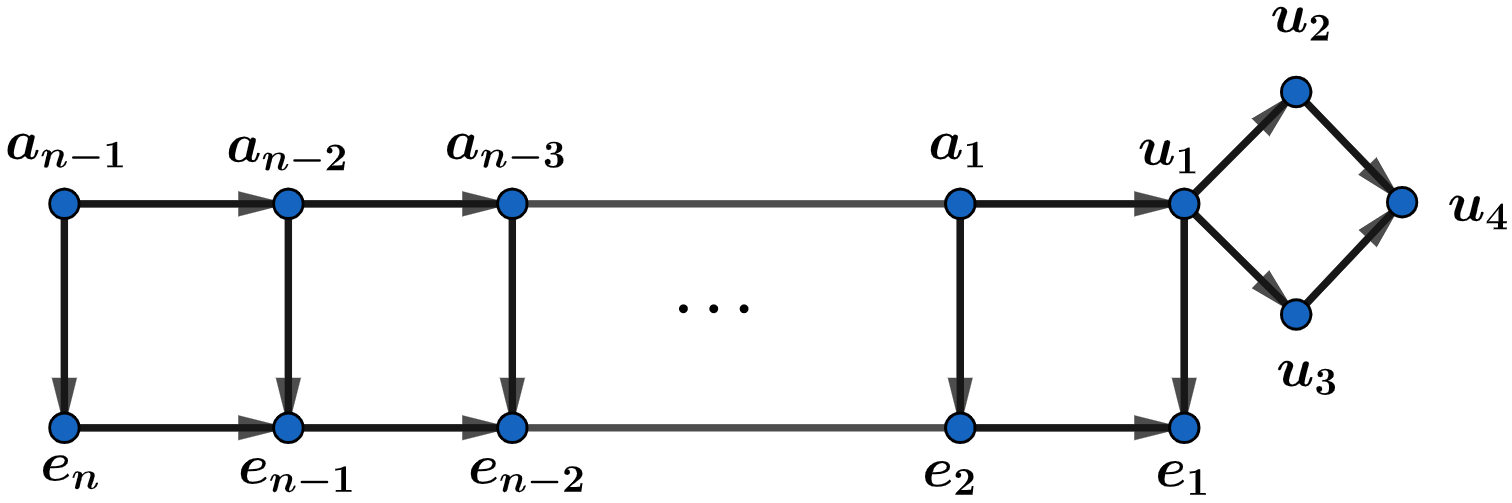}
\caption{The poset $P_4$.}
\label{PlaneP4}
\end{figure}

\begin{figure}[htp]
\centering
\includegraphics[width=10cm,height=4cm]{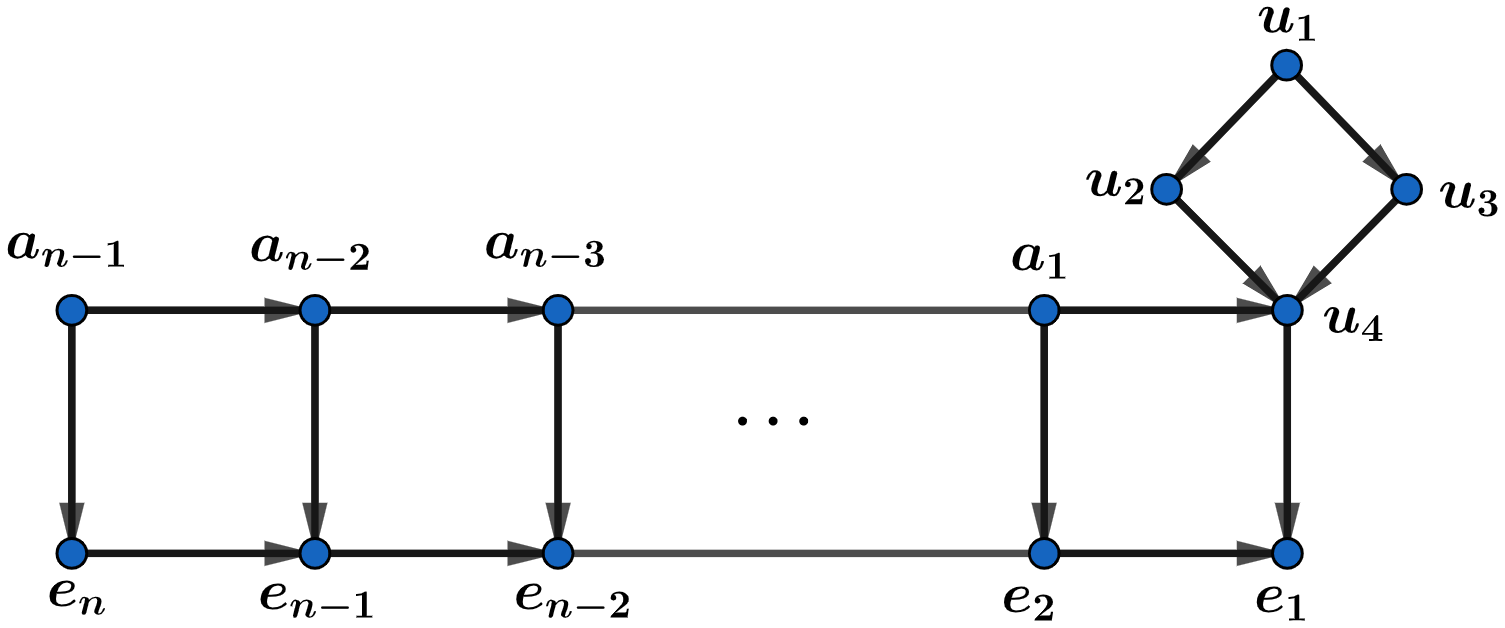}
\caption{The poset $P_5$.}
\label{PlaneP5}
\end{figure}

When $F(x,y;P)=D(x,q,q,q)$, we obtain that the enumerative generating function of the set $\pi(P_4)$ in Figure \ref{PlaneP4} is
\begin{align*}
\mathrm{PF}(\pi(P_4))=\frac{(1+q^{n+1})(1-q^{n+4})-q(1+q^{n+2})(1-q^n)}{(q;q)_n(q;q)_{n+4}}.
\end{align*}
The number of linear extensions of $P_4$ is $e(P_4)=\frac{8\cdot(2n+3)!}{n!(n+4)!}$.
When $F(x,y;P)=D(q,q,q,x)$, we obtain that the enumerative generating function of the set $\pi(P_5)$ in Figure \ref{PlaneP5} is
\begin{align*}
\mathrm{PF}(\pi(P_5))=\frac{1+q^2}{(q;q)_n(q;q)_{n-1}(1-q^2)(1-q^3)(1-q^{n+3})(1-q^{n+4})}.
\end{align*}
The number of linear extensions of $P_5$ is $e(P_5)=\frac{(2n+3)!}{3\cdot n!(n-1)!(n+3)(n+4)}$.

\begin{figure}[htp]
\centering
\includegraphics[width=5cm,height=3cm]{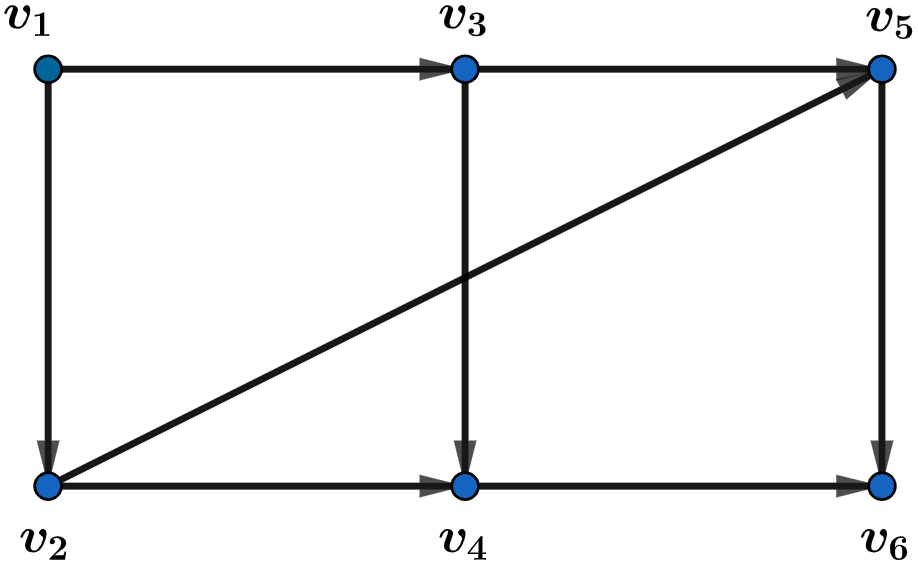}
\caption{The diamond poset with diagonal.}
\label{Square2}
\end{figure}

\begin{figure}[htp]
\centering
\includegraphics[width=12cm,height=4cm]{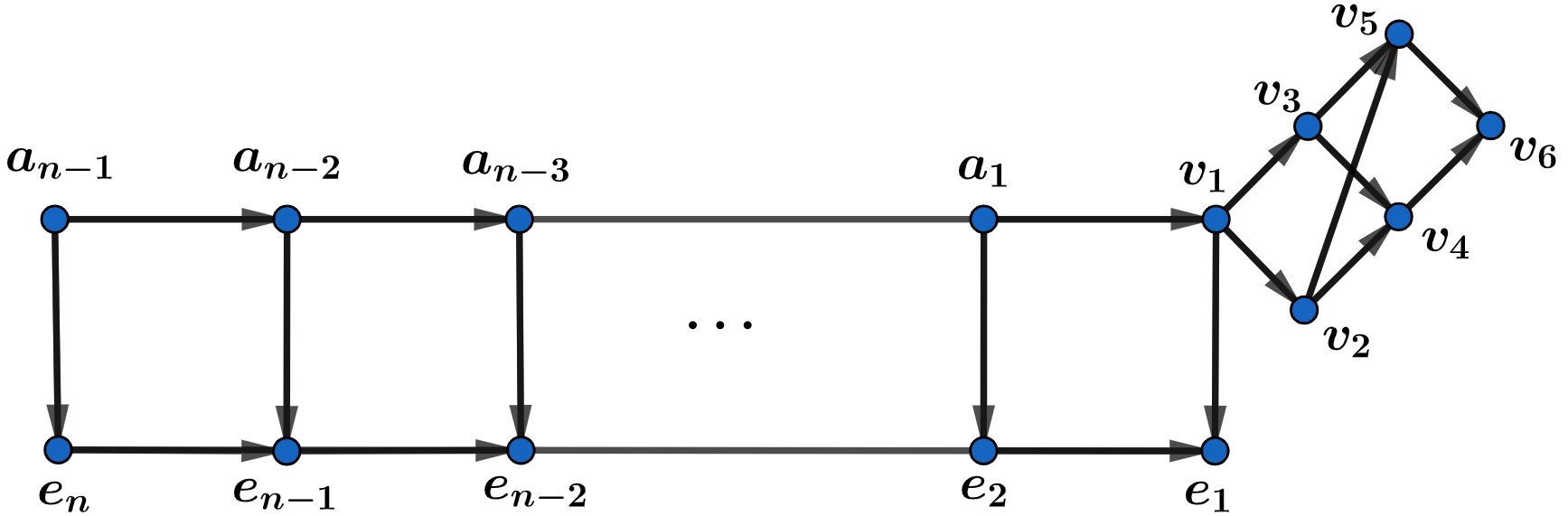}
\caption{The poset $P_6$.}
\label{PlaneP6}
\end{figure}

\begin{figure}[htp]
\centering
\includegraphics[width=12cm,height=5cm]{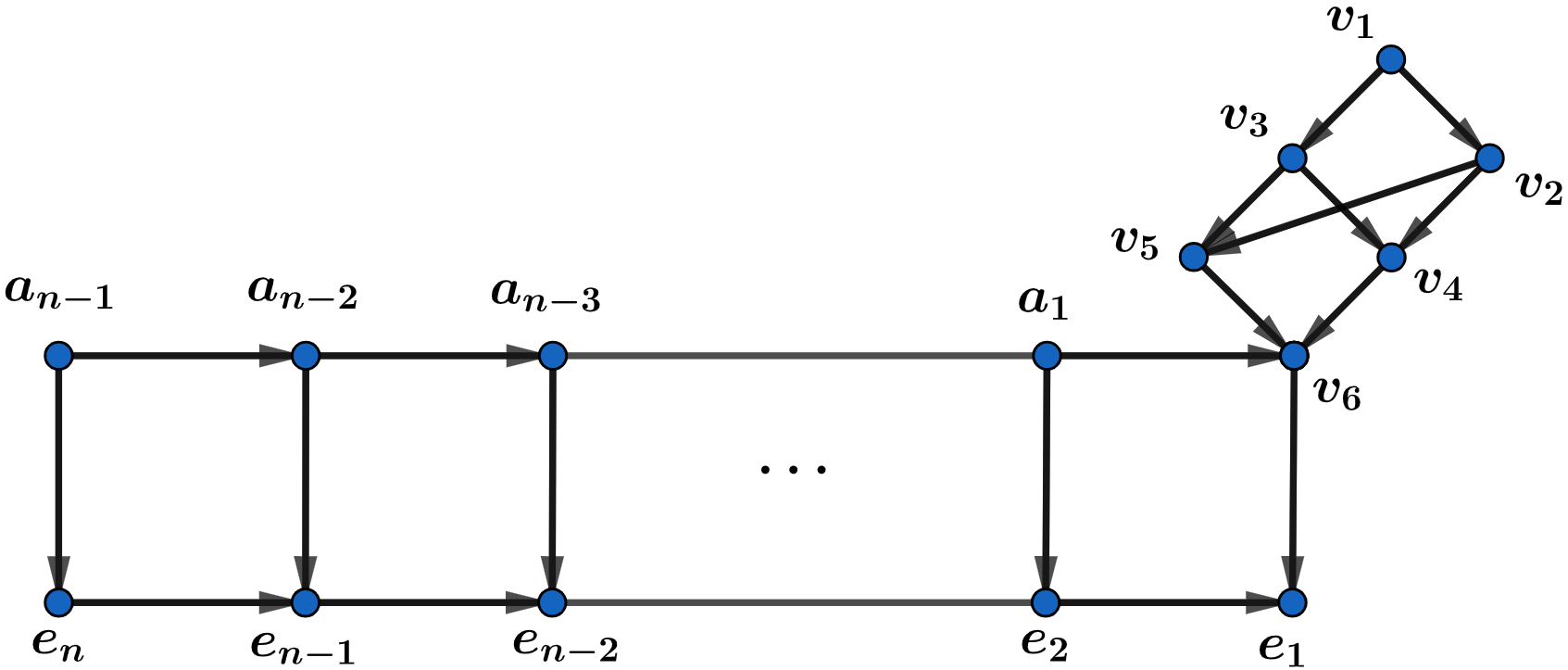}
\caption{The poset $P_7$.}
\label{PlaneP7}
\end{figure}

In \cite{Andrews10}, Andrews, Paule, and Riese obtained the following full generating function of the diamond poset with diagonal in Figure \ref{Square2}:
\begin{footnotesize}
\begin{align*}
&H(x_1,x_2,x_3,x_4,x_5,x_6)=\sum_{v_1\geq v_3\geq v_5\geq v_6\geq 0;\ v_2\geq v_4\geq v_6; \atop v_1\geq v_2;\ v_3\geq v_4;\ v_2\geq v_5}x_1^{v_1}x_2^{v_2}x_3^{v_3}x_4^{v_4}x_5^{v_5}x_6^{v_6}
\\=&\frac{(1-x_1^2x_2x_3)(1-x_1^2x_2^2x_3^2x_4x_5)}{(1-x_1)(1-x_1x_2)(1-x_1x_3)(1-x_1x_2x_3)(1-x_1x_2x_3x_4)
(1-x_1x_2x_3x_5)(1-x_1x_2x_3x_4x_5)(1-x_1x_2x_3x_4x_5x_6)}.
\end{align*}
\end{footnotesize}
When $F(x,y;P)=H(x,q,q,q,q,q)$, we obtain that the enumerative generating function of the set $\pi(P_6)$ in Figure \ref{PlaneP6} is
\begin{align*}
\mathrm{PF}(\pi(P_6))&=\frac{(1-q^{2n+2})(1-q^{2n+6})(1-q^{n+2})(1-q^{n+4})(1-q^{n+6})}{(q;q)_{n+6}(q;q)_{n+4}}
\\ &\ \ \ +\frac{-q(1-q^{2n+4})(1-q^{2n+8})(1-q^n)(1-q^{n+1})(1-q^{n+3})}{(q;q)_{n+6}(q;q)_{n+4}} .
\end{align*}
The number of linear extensions of $P_6$ is
$$e(P_6)=\frac{24\cdot (2n+5)!}{n!(n+6)!}.$$
When $F(x,y;P)=H(q,q,q,q,q,x)$, we obtain that the enumerative generating function of the set $\pi(P_7)$ in Figure \ref{PlaneP7} is
\begin{align*}
\mathrm{PF}(\pi(P_7))=\frac{1-q+q^4-q^5}{(q;q)_n(q;q)_{n-1}(1-q)(1-q^2)^2(1-q^3)(1-q^5)(1-q^{n+5})(1-q^{n+6})}.
\end{align*}
The number of linear extensions of $P_7$ is
$$e(P_7)=\frac{(2n+5)!}{30(n+5)(n+6)\cdot n!(n-1)!}.$$

\subsection{An Extension of $V$-poset}

Assuming that the set $\{h_1,h_2,\ldots,h_k\}$ satisfies the following order relationship:
$$h_1\succ h_2\succ \cdots \succ h_r \prec \cdots \prec h_{k-1} \prec h_k$$
for a certain $1\leq r\leq k$. This poset is called \emph{$V$-poset}.
The Hasse diagram of the $V$-poset is a $V$-shape with only one minimal element.

Now let's consider the poset $P_8$ in Figure \ref{PlaneMN}. This is an extension of $V$-poset.
We calculate the enumerative generating function of the set $\pi(P_8)$.

\begin{figure}[htp]
\centering
\includegraphics[width=10cm,height=9cm]{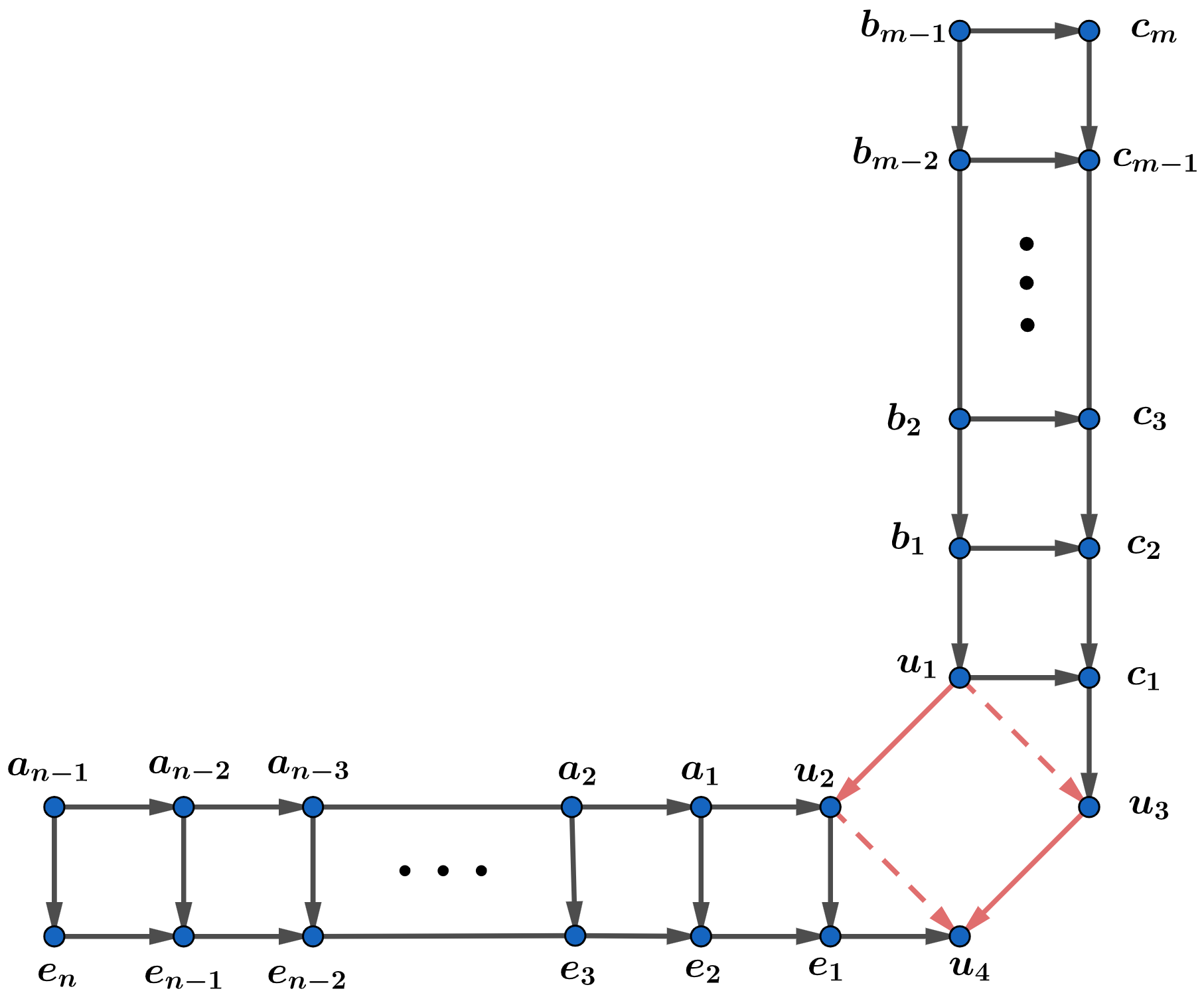}
\caption{The poset $P_8$.}
\label{PlaneMN}
\end{figure}

\begin{cor}
The enumerative generating function of the set $\pi(P_8)$ in Figure \ref{PlaneMN} is
\begin{align*}
\mathrm{PF}(\pi(P_8))&=\frac{1}{(q;q)_n(q;q)_{n-1}(q;q)_m(q;q)_{m-1}}\cdot(T_1+T_2),
\end{align*}
where
$$T_1=\frac{(1-q^{3m+n+1})(1-q^{m+1})(1-q^{m+n+1})
-q(1-q^{3m+n+2})(1-q^m)(1-q^{m+n})}{(1-q^m)(1-q^{m+1})(1-q^{m+n})(1-q^{m+n+1})(1-q^{2m+1})(1-q^{2m+n+1})(1-q^{2m+2n+2})}$$
and
$$T_2=\frac{q^2(1-q^{3m+n+3})(1-q^m)(1-q^{m+n+1})-q(1-q^{3m+n+2})(1-q^{m+1})(1-q^{m+n+2})}{(1-q^m)(1-q^{m+1})
(1-q^{m+n+1})(1-q^{m+n+2})(1-q^{2m+1})(1-q^{2m+n+2})(1-q^{2m+2n+2})}.$$
Furthermore, the number of linear extensions of $P_8$ is
\begin{small}
\begin{align*}
e(P_8)=\frac{(2n+2m+2)!(17m^3+(19n+35)m^2+(7n^2+26n+22)m+(n+1)(n+2)^2)}{n!m!(n-1)!(m+1)!\cdot 2(m+n)(m+n+1)^2(2m+1)(2m+n+1)(m+n+2)(2m+n+2)}.
\end{align*}
\end{small}
\end{cor}
\begin{proof}
When $P$ is the poset in Figure \ref{Square1}, by \eqref{DiamondDDXY},
we have the generating function
\begin{align*}
D(x_1,x_2,x_3,x_4)=\frac{1-x_1^2x_2x_3}{(1-x_1)(1-x_1x_2)(1-x_1x_3)(1-x_1x_2x_3)(1-x_1x_2x_3x_4)}.
\end{align*}
We first take $F(x,y;P)=D(x_1,x,x_3,y).$
By \eqref{FormuPF}, we obtain
\begin{align*}
\overline{D}(x_1,x_3;P):=&\psi_{z_n}\odot \varphi_{z_{n-1}}\odot \varphi_{z_{n-2}}\odot \cdots\odot \varphi_{z_1}\odot F(x,y;P)\big|_{x=z_1=\cdots=z_n=q}
\\=&\frac{1}{(q;q)_n(q;q)_{n-1}}\bigg(\frac{1-x_1^2x_3q^n}{(1-x_1)(1-x_1q^n)(1-x_1x_3)(1-x_1x_3q^n)(1-x_1x_3q^{2n+1})}
\\ &\ \ \ \ \ \ \ \ \ \ -\frac{q(1-x_1^2x_3q^{n+1})}{(1-x_1)(1-x_1q^{n+1})(1-x_1x_3)(1-x_1x_3q^{n+1})(1-x_1x_3q^{2n+1})}\bigg).
\end{align*}
By \eqref{FormuPF} again, we obtain
\begin{align*}
\mathrm{PF}(\pi(P_8))&=\psi_{z_m}\odot \varphi_{z_{m-1}}\odot \varphi_{z_{m-2}}\odot \cdots\odot \varphi_{z_1}\odot \overline{D}(x,y;P)\big|_{x=z_1=\cdots=z_m=q}
\\&=\frac{1}{(q;q)_n(q;q)_{n-1}(q;q)_m(q;q)_{m-1}}(T_1+T_2).
\end{align*}
By \eqref{linearExtension}, we obtain the $e(P_8)$.
This completes the proof.
\end{proof}

\subsection{An Extension of the Ladder Poset}

When $k=r=0$ in Figure \ref{PlaneP8}, the poset is called the \emph{ladder poset} in \cite{BellBraun22}.

Now let's consider the poset $P_9$ in Figure \ref{PlaneP8}.
We mainly calculate the enumerative generating function of the set $\pi(P_9)$. Note that we allow $k\leq r$ in Figure \ref{PlaneP8}.

\begin{figure}[htp]
\centering
\includegraphics[width=16cm,height=3cm]{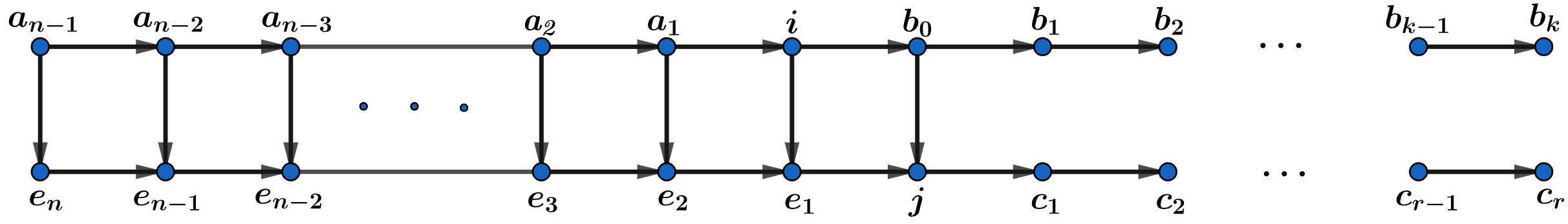}
\caption{The poset $P_9$.}
\label{PlaneP8}
\end{figure}

\begin{figure}[htp]
\centering
\includegraphics[width=9cm,height=3cm]{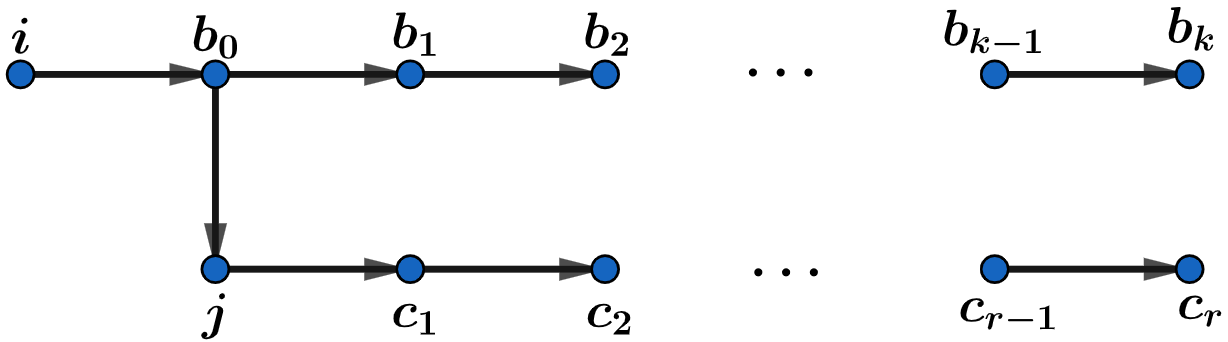}
\caption{The partial graph of poset $P_9$.}
\label{TwoRK}
\end{figure}

\begin{cor}
The enumerative generating function of the set $\pi(P_9)$ in Figure \ref{PlaneP8} is
\begin{align*}
\mathrm{PF}(\pi(P_9))&=\frac{1-q-q^{n+k+2}+q^{n+r+2}}{(q;q)_{n+r+1}(q;q)_{n+k+2}}
\\&\ \ +\sum_{i=0}^{r}\frac{(-1)^{i+1}q^{\frac{i(i+3)}{2}+n+1}\Big((1-q^{n+1})(1-q^{n+i})-(1-q^n)(1-q^{n+i+1})\Big)}
{(1-q^{n+i})(1-q^{n+i+1})(q;q)_n(q;q)_{n+1}(q;q)_i(q;q)_{r-i}(q^{2n+i+2};q)_{k+1}}.
\end{align*}
Furthermore, the number of linear extensions of $P_9$ is
\begin{align*}
e(P_9)=&\frac{(k-r+1)\cdot (2n+k+r+2)!}{(n+r+1)!(n+k+2)!}
\\&+\sum_{i=0}^r\frac{(-1)^{i+1}i\cdot (2n+i+1)!(2n+k+r+2)!}{(n+i)(n+i+1)\cdot n!i!(n+1)!(r-i)!(2n+k+i+2)!}.
\end{align*}
\end{cor}
\begin{proof}
When $P$ is the partial graph of $\pi(P_9)$ in Figure \ref{TwoRK},
we have the generating function
\begin{align*}
&F(x,y;P)=\sum_{i\geq b_0\geq b_1\geq \cdots\geq b_k\geq 0; \atop b_0\geq j\geq c_1\geq c_2\geq \cdots\geq c_r\geq 0}x^iy^j q^{b_0+b_1+\cdots+b_k+c_1+\cdots +c_r}
\\=& \frac{1}{(y;q)_{r+1}(x;q)_{k+2}}+\sum_{i=0}^{r} \frac{(-1)^{i+1}yq^{i(i+3)/2}}{(1-x)(1-yq^i)(q;q)_i(q;q)_{r-i}(xyq^{i+1};q)_{k+1}}.
\end{align*}
This formula is derived using MacMahon's partition analysis. See \cite{Omega-MPA} for a proof.
By \eqref{FormuPF}, the enumerative generating function of the set $\pi(P_9)$ in Figure \ref{PlaneP8} is
\begin{align*}
&\mathrm{PF}(\pi(P_9))=\frac{F(q^n,q^{n+1};P)-qF(q^{n+1},q^n;P)}{(q;q)_n (q;q)_{n-1}}
\\=&\frac{1}{(q;q)_n(q;q)_{n-1}}\left(\frac{1}{(q^{n+1};q)_{r+1}(q^{n};q)_{k+2}}
-\frac{q}{(q^n;q)_{r+1}(q^{n+1};q)_{k+2}}\right)
\\&\ \ \ +\sum_{i=0}^{r}\frac{q^{n+1}(-1)^{i+1}q^{i(i+3)/2}\Big((1-q^{n+1})(1-q^{n+i})-(1-q^n)(1-q^{n+i+1})\Big)}
{(q;q)_n(q;q)_{n+1}(1-q^{n+i+1})(1-q^{n+i})
(q;q)_i(q;q)_{r-i}(q^{2n+i+2};q)_{k+1}}
\\=&\frac{1-q-q^{n+k+2}+q^{n+r+2}}{(q;q)_{n+r+1}(q;q)_{n+k+2}}
\\&\ \ \ +\sum_{i=0}^{r}\frac{(-1)^{i+1}q^{\frac{i(i+3)}{2}+n+1}\Big((1-q^{n+1})(1-q^{n+i})-(1-q^n)(1-q^{n+i+1})\Big)}
{(1-q^{n+i})(1-q^{n+i+1})(q;q)_n(q;q)_{n+1}(q;q)_i(q;q)_{r-i}(q^{2n+i+2};q)_{k+1}}.
\end{align*}
By \eqref{linearExtension}, we obtain the $e(P_9)$.
This completes the proof.
\end{proof}

\section{Concluding Remark}
Our main contribution is Theorem \ref{FullTwoPlane}, which extends $P$-partitions by two-rowed plane partitions.
This idea enables us to obtain explicit expressions for enumerative generating functions and linear extensions of some posets.

Our next project is to extend $P$-partitions by \emph{$r$-rowed plane partitions}. Here $r$-rowed plane partitions are the rectangular plane partitions with at most $r$ rows and $c$ columns. We hope to construct operators similar to $\varphi_{z}$ and $\psi_{z}$, and find their combinatorial interpretations. This may 
help us better understand Gansner's elegant formula in Theorem \ref{Theorem-Gansner}. Our approach will extend and enrich the theory of $P$-partitions.

{\small \textbf{Acknowledgements:}
The authors would like to express their sincere appreciation for all suggestions for improving the presentation of this paper.
This work was partially supported by the National Natural Science Foundation of China [12071311].

\end{document}